\documentclass[reqno]{amsart}

\usepackage[Symbol]{upgreek}

\usepackage{amssymb}
\usepackage{epic}
\usepackage{mathrsfs}
\usepackage{accents}
\usepackage{stmaryrd}

\input{xy}
\xyoption{matrix}
\xyoption{arrow}

\newcommand{\bbold}{\mathbb}

\def\R { {\bbold R} }

\def\N { {\bbold N} }
\def\T { {\bbold T} }

\def \Ko{\operatorname{Ko}}

\def \ex{\operatorname{e}}

\renewcommand\epsilon{\varepsilon}

\def \d{\operatorname{d}}

\def \<{\langle}
\def \>{\rangle}

\def \((  {(\!(}
\def \)) {)\!)}

\DeclareMathSymbol{\precequ}{\mathrel}{symbols}{"16}
\DeclareMathSymbol{\succequ}{\mathrel}{symbols}{"17}

\newtheorem{theorem}{Theorem}[section]
\newtheorem{lemma}[theorem]{Lemma}
\newtheorem{prop}[theorem]{Proposition}
\newtheorem{cor}[theorem]{Corollary}

\theoremstyle{definition}

\theoremstyle{remark}

\def \Zero{\operatorname{Z}}

\let\oldi\i
\let\oldj\j
\renewcommand\i{\relax\ifmmode{\boldsymbol{i}}\else\oldi\fi}
\renewcommand\j{\relax\ifmmode{\boldsymbol{j}}\else\oldj\fi}

\renewcommand\leq{\leqslant}
\renewcommand\geq{\geqslant}
\renewcommand\preceq{\preccurlyeq}

\renewcommand\le{\leq}
\renewcommand\ge{\geq}
\renewcommand\frak{\mathfrak}

\DeclareMathAlphabet{\mathbf}{OML}{cmm}{b}{it}

\DeclareFontFamily{U}{fsy}{}
\DeclareFontShape{U}{fsy}{m}{n}{<->s*[.9]psyr}{}
\DeclareSymbolFont{der@m}{U}{fsy}{m}{n}
\DeclareMathSymbol{\der}{\mathord}{der@m}{182}

\DeclareSymbolFont{der@m}{U}{fsy}{m}{n}
\DeclareMathSymbol{\derdelta}{\mathord}{der@m}{100}

\DeclareSymbolFont{imag@m}{OT1}{cmr}{m}{ui}
\DeclareMathSymbol{\imag}{\mathord}{imag@m}{105}

\DeclareFontFamily{OMS}{smallo}{}
\DeclareFontShape{OMS}{smallo}{m}{n}{<->s*[.65]cmsy10}{}
\DeclareSymbolFont{smallo@m}{OMS}{smallo}{m}{n}
\DeclareMathSymbol{\smallo}{\mathord}{smallo@m}{79}

\DeclareFontFamily{OMS}{largerdot}{}
\DeclareFontShape{OMS}{largerdot}{m}{n}{<->s*[.8]cmsy10}{}
\DeclareSymbolFont{largerdot@m}{OMS}{largerdot}{m}{n}
\DeclareMathSymbol{\largerdot}{\mathord}{largerdot@m}{15}

\DeclareMathSymbol{\llambda}{\mathord}{der@m}{108}
\DeclareMathSymbol{\rrho}{\mathord}{der@m}{114}

\def \Upl{\Uplambda}
\def \upo{\upomega}
\def \Upo{\Upomega}

\newcommand{\equationqed}[1]{\[\pushQED{\qed}#1 \qedhere\popQED\]\let\qed\relax}
\newcommand{\alignqed}[1]{\begin{align*}\pushQED{\qed} #1 \qedhere\popQED\end{align*}\let\qed\relax}

\makeatletter
\newcommand{\dminus}{\mathbin{\text{\@dminus}}}

\newcommand{\@dminus}{%
  \ooalign{\hidewidth\raise1ex\hbox{\bf.}\hidewidth\cr$\m@th-$\cr}%
}
\makeatother

\begin{document}

\title{Dimension in the Realm of Transseries}

\author[Aschenbrenner]{Matthias Aschenbrenner}
\address{Department of Mathematics\\
University of California, Los Angeles\\
Los Angeles, CA 90095\\
U.S.A.}
\email{matthias@math.ucla.edu}

\author[van den Dries]{Lou van den Dries}
\address{Department of Mathematics\\
University of Illinois at Urbana-Cham\-paign\\
Urbana, IL 61801\\
U.S.A.}
\email{vddries@math.uiuc.edu}

\author[van der Hoeven]{Joris van der Hoeven}
\address{\'Ecole Polytechnique\\
91128 Palaiseau Cedex\\
France}
\email{vdhoeven@lix.polytechnique.fr}

\begin{abstract} Let $\T$ be the differential field of transseries. We 
establish some basic properties of the {\em dimension\/} of a definable 
subset of $\T^n$, also in relation to its {\em codimension\/} in the ambient space
$\T^n$. The case of dimension $0$ is of special interest, and
can be characterized both in topological terms (discreteness)
and in terms of the Herwig-Hrushovski-Macpherson notion of 
co-analyzability.  
\end{abstract}

\date{December 2016}


\maketitle
  
\section*{Introduction} 

\noindent
The field of Laurent series with real coefficients comes with a natural derivation but is too small to be closed under integration and exponentiation. These defects are cured by passing to a certain canonical extension, the ordered differential field~$\T$ of transseries. Transseries are formal
series in an indeterminate $x>\R$, such as
\begin{align*}
   &   - 3 \ex^{\ex^x} + \ex^{\textstyle\frac{\ex^x}{\log x} +
  \frac{\ex^x}{\log^2 x} + \frac{\ex^x}{\log^3 x} + \cdots} - x^{11} + 7\\
  &   \hspace{2em} + \frac{\pi}{x} + \frac{1}{x \log x} + \frac{1}{x \log^2
  x} + \frac{1}{x \log^3 x} + \cdots \nonumber\\
  &   \hspace{2em} + \frac{2}{x^2} + \frac{6}{x^3} + \frac{24}{x^4} +
  \frac{120}{x^5} + \frac{720}{x^6} + \cdots \nonumber\\
  &   \hspace{2em} + \ex^{- x} + 2 \ex^{- x^2} + 3 \ex^{- x^3} + 4 \ex^{-
  x^4} + \cdots, \nonumber
\end{align*}
where $\log^2 x := (\log x)^2$, etc.~Transseries, that is, 
elements of $\T$,  
are also the {\em lo\-ga\-rith\-mic-exponential series\/}  ($\operatorname{LE}$-series, for short)
from \cite{DMM}; we refer to that paper, or to Appendix~A of our book~\cite{ADH}, for a detailed construction of $\T$.

What we need for now is that $\T$ is a real closed field extension
of the field $\R$ of real numbers and that $\T$ comes equipped with 
a distinguished element $x>\R$, an exponential operation $\exp\colon \T \to \T$ and a distinguished derivation $\der\colon \T\to \T$. The exponentiation here is an isomorphism of the ordered additive group of $\T$ onto the ordered multiplicative group
$\T^{>}$ of positive elements of $\T$. The derivation~$\der$
comes from
differentiating a transseries termwise with respect to $x$, and
we set $f':= \der(f)$, $f'':= \der^2(f)$, and so on, for $f\in \T$;
in particular, $x'=1$, and $\der$
is compatible with exponentiation: $\exp(f)'=f'\exp(f)$ for $f\in \T$. Moreover, the constant field of $\T$ is~$\R$, that is,
$\{f\in \T:\ f'=0\}=\R$; see again \cite{ADH} for details.

In Section~\ref{dad} we define for any differential field $K$ (of characteristic $0$ in this paper) and any set $S\subseteq K^n$ its (differential-algebraic)
dimension 
$$\dim S\in \{-\infty,0,1,\dots,n\} \quad (\text{with $\dim S=-\infty$ iff $S=\emptyset$}).$$ 
Some dimension properties hold in this generality, but for more substantial results we assume that $K=\T$ and $S$ is {\em definable\/} in $\T$, in which case we have: 
$$\dim S = n\ \Longleftrightarrow\ S \text{ has nonempty interior in }\T^n.$$ 
Here $\T$ is equipped with its order topology, and each $\T^n$
with the corresponding product topology. This equivalence is shown in
Section~\ref{T}, where we also prove:

\begin{theorem}\label{ddim} If $S\subseteq \T^m$ and $f\colon S \to \T^n$ are definable, then $\dim S\ge \dim f(S)$, for every $i\in \{0,\dots,m\}$ the set $B(i):=\big\{y\in \T^n:\ \dim f^{-1}(y) =i\big\}$ is definable, and $\dim f^{-1}\big(B(i)\big)=i+\dim B(i)$.
\end{theorem} 

\noindent
In Section~\ref{d0d} we show that for definable nonempty $S\subseteq \T^n$,
$$ \dim S =0\ \Longleftrightarrow\  \text{$S$ is discrete.}$$
For $S\subseteq \T^n$ to be discrete means as usual that every point of $S$ has a
neighborhood in $\T^n$ that contains no other point of $S$.
For example, $\R^n$ as a subset of $\T^n$ is discrete! Proving the
backwards direction of the equivalence above involves an unusual cardinality argument. Both directions use key results from~\cite{ADH}. 

\medskip\noindent
The rest of the paper is inspired by \cite[Theorem 16.0.3]{ADH}, which suggests 
that for a definable set $S\subseteq \T^n$ to have dimension $0$ amounts to
$S$ being controlled in some fashion by the constant field $\R$. In what fashion? Our first guess was that perhaps every
definable subset of $\T^n$ 
of dimension $0$ is the image of some definable map~${\R^m \to \T^n}$. 
(Every such image has indeed dimension $0$.)
 It turns out, however, that the solution 
set of the algebraic differential equation $yy''=(y')^2$ in $\T$, which has dimension $0$, is {\em not\/} such an 
image: in Section~\ref{sec:pc} we show how this follows from a fact
about automorphisms of $\T$ to be
established in \cite{ADH1}. (In that section we call an image as above
{\em parametrizable by constants}; we have since learned that it already has a name in the literature, namely, {\em internal to the constants}, a special case of a general model-theoretic notion; see~\cite[Section~7.3]{Pillay}.)

The correct way to understand the model-theoretic meaning of 
dimension~$0$ is the concept of {\em co-analyzability}
from \cite{HHM}. 
This is the topic of
Section~\ref{sec:ca}, where we also answer positively a question that
partly motivated our paper:
given definable $S\subseteq \T^{m}$  and definable $f\colon S \to \T^n$, does 
there always exist
an $e\in \N$ such that $|f^{-1}(y)|\le e$ 
for all $y\in \T^n$ for which $f^{-1}(y)$ is finite? In other words, is the
quantifier ``there exist infinitely many'' available for free?

  We thank James Freitag for pointing us to the notion of co-analyzability.

\section{Differential-algebraic Dimension}~\label{dad}

\noindent
We summarize here parts of subsection~2.25 in~\cite{vdD89},
referring to that paper for proofs.
Throughout this section $K$ is a differential field (of characteristic zero with a single distinguished derivation, in this paper), with constant field $C\ne K$. Also, $Y=(Y_1,\dots, Y_n)$ is a tuple of distinct differential indeterminates, and $K\{Y\}$ the ring of differential polynomials in $Y$ over $K$.

\subsection*{Generalities}
Let a set $S\subseteq K^n$ be given. Then the differential polynomials
$P_1,\dots, P_m\in K\{Y\}$ are said to be
{\bf $\d$-algebraically dependent on $S$\/} if for some
nonzero differential polynomial $F\in K\{X_1,\dots, X_m\}$,
$$F\big(P_1(y),\dots, P_m(y)\big)\ =\ 0\ \text{ for all }y=(y_1,\dots, y_n)\in S;$$ 
if no such $F$ exists, we say that $P_1,\dots, P_m$ are
{\bf $\d$-algebraically independent on $S$,} and in that case we must have $m\le n$; the prefix $\d$ stands for {\em differential.}\/ For nonempty $S$ we define the (differential-algebraic) {\bf dimension}
$\dim S$ of $S$ to be the largest $m$ for which there exist $P_1,\dots, P_m\in K\{Y\}$ that are $\d$-algebraically independent
on $S$, and if $S=\emptyset$, then we set $\dim S:= -\infty$. 

In particular, for nonempty $S$, $\dim S =0$ means that for every 
$P\in K\{Y\}$ there
exists a nonzero $F\in K\{X\}$, $X=X_1$, such that $F\big(P(y)\big)=0$ for all $y\in S$. 
As an example, let $a\in K^n$ and consider $S=\{a\}$. For $P\in K\{Y\}$ 
we have $F(P(a))=0$ for $F(X):=X-P(a)$, so $\dim \{a\} =0$. Also, $\dim C^n =0$ by Lemma~\ref{dim1}.

Of course, this notion of dimension is relative to $K$, and if we need to indicate the ambient $K$ we write~$\dim_K S$ instead of 
$\dim S$. But this will hardly be necessary, since $\dim_K S = \dim_L S$ for any differential field extension $L$ of $K$. 

Below we also consider the structure
$(K,S)$: the differential field $K$ equipped with the $n$-ary
relation $S$. 
The following is a useful characterization of dimension in terms
of {\em differential transcendence degree} (for which see \cite[Section 4.1]{ADH}):

\begin{lemma}\label{dim1} Let $(K^*, S^*)$ be a $|K|^+$-saturated elementary
extension of $(K,S)$ and assume $S$ is not empty. Then
$$\dim_K S\ =\ \max\!\big\{\text{differential transcendence degree of $K\<s\>$ over $K$}:\ s\in S^*\big\}.$$
\end{lemma}

\noindent
Here are some easy consequences of the definition of {\em dimension}\/
and Lemma~\ref{dim1}:

\begin{lemma}\label{dim2} Let $S, S_1, S_2\subseteq K^n$. Then: \begin{enumerate}
\item[(i)] if $S$ is finite and nonempty, then $\dim S =0$; $\dim K^n=n$;
\item[(ii)] $\dim S < n\ \Longleftrightarrow\ S\subseteq \big\{y\in K^n:\ P(y)=0\big\}$ for some nonzero $P\in K\{Y\}$;
\item[(iii)] $\dim(S_1\cup S_2)=\max(\dim S_1, \dim S_2)$;
\item[(iv)] $\dim S^{\sigma}=\dim S$ for each permutation $\sigma$ of $\{1,\dots,n\}$, where
$$S^\sigma\ :=\ \big\{\big(y_{\sigma(1)},\dots,y_{\sigma(n)}\big):\ (y_1,\dots,y_n)\in S\big\};$$
\item[(v)] if $m\le n$ and $\pi\colon K^n \to K^m$ is given by $\pi(y_1,\dots, y_n)=(y_1,\dots, y_m)$, then $\dim \pi(S) \le \dim S$;
\item[(vi)] if $\dim S=m$, then $\dim \pi(S^{\sigma})=m$ for some
$\sigma$ as in \textup{(iv)} and $\pi$ as in \textup{(v)}. 
\end{enumerate}
\end{lemma}

\noindent
The next two lemmas are not in \cite{vdD89}, and are left as
easy exercises: 

\begin{lemma}\label{dim3} $\dim(S_1\times S_2)=\dim S_1 + \dim S_2$ for $S_1\subseteq K^m$ and $S_2\subseteq K^n$. 
\end{lemma}


\begin{lemma}\label{dim4} $\dim_K S = \dim_{K^*} S^*$ in the situation of Lemma~\ref{dim1}.
\end{lemma}

\noindent
Let now $K^*$ be any elementary extension of $K$ and suppose $S$ is definable in $K$, say by the formula $\phi(y_1,\dots, y_n)$ in the language of differential fields with names for the elements of $K$. Let $S^*\subseteq (K^*)^n$ be defined in $K^*$ by the same formula~$\phi(y_1,\dots, y_n)$. Note that $S^*$ does not depend on the choice of $\phi$. We have the following easy consequence of 
Lemma~\ref{dim4}:

\begin{cor} $\dim_K S = \dim_{K^*} S^*$.
\end{cor}

\subsection*{Differential boundedness} For a set $S\subseteq K^{n+1}$ and $y\in K^n$ we define $$S(y)\ :=\ \big\{z\in K:\ (y,z)\in S\big\}  \qquad\text{(the section of $S$ above $y$).}$$ 
We say that $K$ is {\bf $\d$-bounded} if for every definable set $S\subseteq K^{n+1}$ there exist $P_1,\dots, P_m\in K\{Y,Z\}$ (with $Z$ an extra indeterminate) such that if $y\in K^n$ and 
$\dim S(y)=0$, then $S(y)\subseteq \{z\in K:\ P_i(y,z)=0\}$
for some $i\in \{1,\dots,m\}$ with $P_i(y,Z)\ne 0$. (In view of
Lemma~\ref{dim2}(ii), this is equivalent to the differential field $K$ being {\em differentially bounded\/} as defined on p.~203 of~\cite{vdD89}.) Here is the main consequence of $\d$-boundedness, taken from \cite{vdD89}:

\begin{prop}\label{Kddim} Assume $K$ is $\d$-bounded. Let $S\subseteq K^m$ and $f\colon S \to K^n$ be definable. Then $\dim S\ge \dim f(S)$. Moreover, for every $i\in \{0,\dots,m\}$ the set $B(i):=\big\{y\in K^n:\ \dim f^{-1}(y) =i\big\}$ is definable, and $\dim f^{-1}\big(B(i)\big)=i+\dim B(i)$.
\end{prop}

\noindent
As $\T$ is $\d$-bounded (see Section~\ref{T}), this 
gives Theorem~\ref{ddim}. Differentially closed fields are $\d$-bounded, as pointed out in \cite{vdD89}.
Guzy and Point~\cite{GP} (see~also~\cite{BMR}) show that existentially closed ordered differential fields, and Scanlon's
$\d$-henselian valued differential fields with many constants  (see \cite[Chapter~8]{ADH}) are $\d$-bounded.

\section{Dimension and Codimension}\label{dimcodim}

\noindent
This section will not be used in the rest of this paper, but is included
for its own sake. The main
result is Corollary~\ref{dico}. A byproduct of the treatment here is a 
simpler proof of \cite[Theorem 5.9.1]{ADH} that avoids the nontrivial facts about regular local rings used in \cite{ADH},
where we followed closely Johnson's proof in \cite{JJ} of a
more general result.

\bigskip\noindent
Let $y=(y_1,\dots,y_n)$ be a tuple of elements of a differential field extension of $K$, and
let $d$ be the differential transcendence degree of $F:=K\<y\>$ over $K$: there are~${i_1< \dots < i_d}$ in $\{1,\dots,n\}$ such that $y_{i_1},\dots,y_{i_d}$ are $\d$-algebraically independent over~$K$, 
but there are no $i_1< \dots < i_d< i_{d+1}$ in $\{1,\dots,n\}$ such that $y_{i_1},\dots,y_{i_d}, y_{i_{d+1}}$ are $\d$-algebraically independent over $K$. We wish to characterize  $d$ alternatively as follows: there should exist~${n-d}$ ``independent'' relations
$P_1(y)=\dots = P_{n-d}(y)=0$, with all $P_i\in K\{Y\}$, but not more than $n-d$
such relations. 
The issue here   
is what ``independent'' should mean. 

We say that a $\d$-polynomial $P\in K\{Y\}$ has order at most $\vec r=(r_1,\dots, r_n)\in \N^n$ if
$P\in K\big[Y_j^{(r)}:\ 1\le j\le n,\ 0\le r\le r_j\big]$.
Given $P_1,\dots, P_m\in K\{Y\}$ of order at most $\vec r\in \N^n$, consider the $m\times n$-matrix over 
$F$ with $i,j$-entry 
$$     \frac{\partial P_i}{\partial Y_j^{(r_j)}}(y) \qquad (i=1,\dots,m,\ j=1,\dots,n).$$
This matrix has rank $\le \min(m,n)$.
We say that $P_1,\dots,P_m$ are {\bf strongly $\d$-in\-de\-pen\-dent at $y$} if 
for some  $\vec{r}\in\N^n$ with $P_1,\dots,P_m$ of order at most~$\vec r$, this matrix has rank $m$; thus $m\le n$ in that case.

Set $R:=K\{Y\}$ and $\frak{p}:=\big\{P\in R:P(y)=0\big\}$,  a differential prime ideal of $R$. With these notations we have:

\begin{lemma} There are $P_1,\dots,P_{n-d}\in\frak p$ that are strongly $\d$-independent at $y$.
\end{lemma}
\begin{proof} Set $m:=n-d$ and permute 
indices such that $y_{m+1},\dots, y_n$ is a differential transcendence base
of $F=K\<y\>$ over~$K$. For $i=1,\dots,m$, pick 
$$P_i(Y_i,Y_{m+1},\dots, Y_n)\in K\{Y_i,Y_{m+1},\dots, Y_n\}\ \subseteq\ K\{Y\}$$
such that $P_i(Y_i,y_{m+1},\dots, y_n)$ is a minimal annihilator
of $y_i$ over $K\<y_{m+1},\dots, y_{n}\>$. Let $P_i$ have order $r_i$ in $Y_i$.
Then the minimality of $P_i$ gives $$\frac{\partial P_i}{\partial Y_i^{(r_i)}}(y_i,y_{m+1},\dots, y_n)\ne 0, \qquad (i=1,\dots,m).$$
Next we take $r_{m+1},\dots, r_n\in \N$ such that all $P_i$ have order
$\le r_j$ in $Y_j$ for $j=m+1,\dots,n$. Considering all $P_i$ as elements of
$K\{Y\}$ we see that $P_1,\dots, P_m$ have order $\le (r_1,\dots, r_n)$, and
that the $m\times m$ matrix  
$$\left(\frac{\partial P_i}{\partial Y_j^{(r_j)}}(y)\right) \qquad 
(1\le i, j\le m)$$  
is diagonal, with nonzero determinant. 
\end{proof}

\noindent
We refer to \cite[Section 5.4]{ADH} for what it means for $P_1,\dots, P_m\in R$ to be {\em $\d$-independent at $y$}. By \cite[Lemma~5.4.7]{ADH}, if $P_1,\dots, P_m\in R$ are strongly $\d$-independent at $y$, then they are $\d$-independent at $y$ (but
the converse may fail).
Below we show that if $P_1,\dots, P_m\in \frak{p}$ are $\d$-independent at $y$, then $m\le n-d$. 

\medskip\noindent
The notion of $\d$-independence at $y$
is more intrinsic and more flexible than that of strong $\d$-independence at $y$. To discuss the former in more detail, we need some terminology from ~\cite{ADH}. Let $A$ be a commutative ring, $\frak{p}$ a prime ideal of $A$, and $M$ an $A$-module; then a family $(f_i)$ of elements of $M$ is said to be {\em independent at $\frak{p}$\/}
if the family $(f_i+\frak{p}M)$ of elements of the $A/\frak{p}$-module $M/\frak{p}M$ is linearly independent. Next, let
$A$ also be a differential ring extension of $K$. 
Then the $K$-algebra $A$ yields the $A$-module~$\Omega_{A|K}$ of 
K\"ahler differentials with the (universal) $K$-derivation 
$$a\mapsto \d a\ =\ \d_{A|K} a\ :\ A \to \Omega_{A|K}.$$
Following Johnson~\cite{JJ} we make this $A$-module compatibly into an $A[\der]$-module by $\der (\d a):=\d \der a$ for $a\in A$; a family of elements of
$\Omega_{A|K}$ is said to be {\em $\d$-independent\/} if this family
is linearly independent in $\Omega_{A|K}$ viewed as an $A[\der]$-module. This means for $a_1,\dots, a_m\in A$: the
differentials $\d a_1,\dots, \d a_m\in \Omega_{A|K}$ are $\d$-independent iff the family $\big(\!\d a_i^{(r)}\big)$ ($i=1,\dots,m,\  r=0,1,2,\dots$) is linearly independent in the $A$-module $\Omega_{A|K}$; given also a prime ideal $\frak{p}$ of $A$ we say that $\d a_1,\dots,\d a_m$
are {\em $\d$-independent at $\frak{p}$\/} if the family $(\d a_i^{(r)})$ is independent at $\frak{p}$ in the $A$-module $\Omega_{A|K}$.

Returning to the differential ring extensions $R$ and $F=K\<y\>$ of $K$, the $R[\der]$-module 
$\Omega_{R|K}$ is free on $\d Y_1,\dots, \d Y_n$, by \cite[Lemma 1.8.11]{ADH}. The $F[\der]$-module~$\Omega_{F|K}$
is generated by $\d y_1,\dots, \d y_n$, as shown in \cite[Section 5.9]{ADH}. In \cite[Section 5.3]{ADH} we assign to every finitely generated
$F[\der]$-module $M$ a number $\text{rank}(M)\in \N$, and we have 
$\text{rank}(\Omega_{F|K})=d$ by \cite[Corollary 5.9.3]{ADH}.

The differential ring morphism $P\mapsto P(y)\colon R \to F$
is the identity on $K$, and makes $F\otimes_{R} \Omega_{R|K}$
into an $F[\der]$-module as explained in \cite[Section 5.9]{ADH}.
Note that the kernel of the above differential ring morphism $R \to F$ is the differential prime ideal
$\frak{p}=\{P\in R:\ P(y)=0\}$ of $R$.

\begin{lemma} Suppose  $P_1,\dots,P_m\in\frak p$ are $\d$-independent at $y$. Then
$m\le n-d$.
\end{lemma}
\begin{proof}
We have a surjective $F[\der]$-linear map
$ F\otimes_{R} \Omega_{R|K} \rightarrow \Omega_{F|K}$
sending~${1\otimes \d P}$ to $\d P(y)$ 
for $P\in R$. Note that $1\otimes \d P_1,\dots, 1\otimes \d P_m$
are in the kernel of this map. By the equivalence $(1)\Leftrightarrow (5)$ and Lemma 5.9.4 in \cite{ADH}, 
the $\d$-independence of $P_1,\dots, P_m$ at $y$ gives that $1\otimes \d P_1,\dots, 1\otimes \d P_m\in F\otimes_R \Omega_{R|K}$ are $F[\der]$-independent (meaning: linearly independent in this $F[\der]$-module). Since the $R[\der]$-module $\Omega_{R|K}$
is free on $dY_1,\dots, dY_n$, the $F[\der]$-module $F\otimes_{R} \Omega_{R|K}$ is free on $1\otimes \d Y_1,\dots, 1\otimes \d Y_n$,
and so has rank $n$. To get $m+d\le n$ it remains to use \cite[Corollary~5.9.3]{ADH} and the fact that $\text{rank}(\Omega_{F|K})=d$. 
\end{proof}

\noindent
Combining the previous two lemmas we conclude:

\begin{cor}\label{dico} The codimension $n-d$ can be characterized as follows:
\begin{align*} n-d\ &=\ \max\{m:\text{some $P_1,\dots,P_m\in\frak p$ are $\d$-independent at $y$}\}\\
&=\ \max\{m:\text{some $P_1,\dots,P_m\in\frak p$ are strongly $\d$-independent at $y$}\}.
\end{align*}
\end{cor}

\noindent
This yields a strengthening of Theorem~5.9.1 and its Corollary~5.9.6 in \cite{ADH}:

\begin{cor}
The following are equivalent:
\begin{enumerate}
\item[(i)] $y_1,\dots,y_n$ are $\d$-algebraic over $K$;
\item[(ii)] there exist $P_1,\dots,P_n\in \frak{p}$ that are $\d$-in\-de\-pen\-dent at $y$;
\item[(iii)] there exist $P_1,\dots,P_n\in \frak{p}$ that are are strongly $\d$-in\-de\-pen\-dent at $y$.
\end{enumerate}
\end{cor}

\noindent
To formulate the above in terms of sets $S\subseteq K^n$ we recall that the {\em Kolchin topology on $K^n$} (called
the {\em differential-Zariski topology on $K^n$} in \cite{vdD89}) is the topology on $K^n$ whose closed sets are the sets
$$\big\{y\in K^n:\  P_1(y)=\cdots= P_m(y)=0\big\}\quad (P_1,\dots, P_m\in K\{Y\}).$$
This is a noetherian topology, and so a Kolchin closed subset of
$K^n$ is the union of its finitely many irreducible components.   For $S\subseteq K^n$ we let $S^{\Ko}$ be its Kolchin closure in $K^n$ with respect to the Kolchin topology. Note that
$\dim S = \dim S^{\Ko}$, since for all $P\in K\{Y\}$ we have:
if $P=0$ on $S$ (that is, $P(y)=0$ for all $y\in S$), then 
$P=0$ on $S^{\Ko}$. 

 Suppose $S^{\Ko}$ is irreducible.
A {\bf tuple of $m$ independent re\-la\-tions on $S$} is defined to be a tuple $(P_1,\dots, P_m)\in K\{Y\}^m$  such that
\begin{enumerate}
\item $P_1(y)=\cdots = P_m(y)=0$ for all $y\in S$;
\item $P_1,\dots,P_m$ are $\d$-independent at some $y\in S$.
\end{enumerate}
Similarly we define a {\bf tuple of $m$ strongly independent relations on $S$,} by replacing 
``$\d$-independent'' in (2) by ``strongly $\d$-independent''. Every tuple of strongly independent relations on $S$ is a tuple of independent relations on $S$.
Since~$S^{\Ko}$ is irreducible, 
$$\frak{p}:=\big\{P\in K\{Y\}:\ \text{$P=0$ on $S$}\big\}$$
is  a differential prime ideal of $K\{Y\}$. Letting $K\{y\}=K\{Y\}/\frak{p}$ be the 
corresponding differential $K$-algebra (an integral domain) with $y=(y_1,\dots,y_n)$, $y_i=Y_i+\frak{p}$, for $P\in K\{Y\}$ we have $P(y)=0$ iff $P=0$ on $S$.
So the considerations above applied to $y$ yield for $d:=\dim S$
and irreducible $S^{\Ko}$:

\begin{cor} There is a tuple of $m$ strongly independent relations on $S$ for $m=n-d$, but there is no tuple of $m$ independent relations on $S$ for $m>n-d$. 
\end{cor}

\section{The Case of $\T$}\label{T}

\noindent
The paper~\cite{vdD89} contains an axiomatic framework for
a reasonable notion of dimension for the definable sets in suitable
model-theoretic structures with a topology. In this section we show that as a consequence of~\cite[Chapter~16]{ADH} the relevant axioms are satisfied for $\T$ with its order topology.

To state the necessary facts about $\T$ from \cite{ADH} we recall from that book that an {\it $H$-field}\/ is an ordered differential field $K$ with constant field~$C$ such that: \begin{enumerate}
\item[(H1)] $\der(a)>0$ for all $a\in K$ with $a>C$;
\item[(H2)] $\mathcal{O}=C+\smallo$, where $\mathcal{O}$ is the
convex hull of $C$ in the ordered field $K$, and $\smallo$ is the maximal ideal of the valuation ring $\mathcal{O}$.
\end{enumerate}
Let $K$ be an $H$-field, and let $\mathcal{O}$ and $\smallo$ be as in (H2). Thus $K$ is a valued field with valuation ring $\mathcal{O}$. The valuation topology on $K$ equals its order topology if $C\ne K$. We consider $K$ as an 
$\mathcal{L}$-structure, where 
$$\mathcal{L}\ :=\ \{\,0,\,1,\, {+},\, {-},\, {\times},\, \der,\, {P},\, {\preceq}\,\}$$ is the language of ordered valued differential fields. The symbols $0,\,1,\, {+},\, {-},\, {\times},\, \der$ are interpreted as usual in $K$, and 
$P$ and $\preceq$ encode the ordering and the valuation: 
$$P(a)\ \Longleftrightarrow\ a>0, \qquad a\preceq b\ \Longleftrightarrow a\in \mathcal{O} b\ \qquad(a,b\in K).$$ Given $a\in K$ we also write $a'$ instead of $\der(a)$, and we set $a^\dagger:= a'/a$ for $a\ne 0$. 

The real closed (and thus ordered) differential field $\T$ is an $H$-field, and in \cite{ADH} we showed that it is a model of a model-complete $\mathcal{L}$-theory $T^{\text{nl}}$. The models of the latter are exactly the $H$-fields $K$ satisfying the following (first-order) conditions: \begin{enumerate}
\item $K$ is Liouville closed;
\item $K$ is $\upo$-free;
\item $K$ is newtonian.
\end{enumerate} 
(An $H$-field $K$ is said to be {\it Liouville closed}\/ if it is real closed and for all $a\in K$ there exists $b\in K$ with $a=b'$ and also a $b\in K^\times$ such that $a=b^\dagger$; for the definition of ``$\upo$-free'' and ``newtonian'' we refer to the
Introduction of \cite{ADH}.) Since ``Liouville closed'' includes ``real closed'', the ordering (and thus the valuation ring) of any model of $T^{\text{nl}}$ is definable
in the underlying differential field of the model. 
We shall prove the dimension results in this paper for all models
of $T^{\text{nl}}$: working in this generality
plays a role even when our main interest is in $\T$. So in the rest of this section we fix an arbitrary model $K$ of $T^{\text{nl}}$, that is, {\em $K$ is a Liouville closed $\upo$-free newtonian $H$-field}. Lemma~\ref{dim2}(ii)
and \cite[Corollary~16.6.4]{ADH} yield:  

\begin{cor}\label{defdimn} For definable $S\subseteq K^n$,
$$ \dim S=n\ \Longleftrightarrow\   \text{$S$ has nonempty interior in $K^n$.}$$
\end{cor}

\noindent
To avoid confusion with the  Kolchin topology, we consider $K$ here and 
below as equipped with its order topology, and $K^n$ with the corresponding 
product topology. 
Combining the previous corollary with (iv)--(vi) in Lemma~\ref{dim2} yields a topological characterization of dimension:

\begin{cor}
For nonempty definable $S\subseteq K^n$, 
$\dim S$ is the largest $m\le n$ such that for some permutation $\sigma$ of $\{1,\dots,n\}$, the subset
$\pi_m(S^\sigma)$ of $K^m$ has nonempty interior;
here $\pi_m(x_1,\dots,x_n):=(x_1,\dots,x_m)$ for $(x_1,\dots,x_n)\in K^n$.
\end{cor}

\noindent
In particular, if $S\subseteq K^n$ is semialgebraic in the sense of the 
real closed field $K$, then $\dim S$ agrees with the usual semialgebraic dimension of $S$ over $K$.

\medskip
\noindent
To get that $K$ is $\d$-bounded, we introduce two
key subsets of $K$, namely $\Upl(K)$ and~$\Upo(K)$. They are defined by the following equivalences, for $a\in K$:
\begin{align*}
a\in \Upl(K)\ &\Longleftrightarrow\ a=-y^{\dagger\dagger} \text{ for some $y\succ 1$  in $K$,}\\
a\in \Upo(K)\  &\Longleftrightarrow\ 4y''+ay=0 \text{ for some $y\in K^\times$.}
\end{align*}
To describe these sets more concretely for $K=\T$, set $\ell_0:= x$ and
$\ell_{n+1}:= \log \ell_n$, so $\ell_n$ is the $n$th iterated logarithm of 
$x$ in $\T$. 
Then for $f\in \T$,
\begin{align*}
  f \in \Upl (\T)\ & \Longleftrightarrow\ f\ <\ \frac{1}{\ell_0} + \frac{1}{\ell_0
  \ell_1} + \cdots + \frac{1}{\ell_0 \ell_1 \cdots \ell_n}\quad  \text{ for some $n$,}\\
  f \in \Upo (\T)\ & \Longleftrightarrow\ f\ <\ \frac{1}{\ell_0^2} +
  \frac{1}{\ell_0^2 \ell_1^2} + \cdots + \frac{1}{\ell_0^2 \ell_1^2 \cdots
  \ell_n^2}\quad  \text{ for some $n$,}
\end{align*}
by \cite[Example after~11.8.19; Proposition~11.8.20 and Corollary~11.8.21]{ADH}.
The set~$\Upl(K)$ is closed downward in $K$: 
if $a\in K$ and $a<b\in \Upl(K)$, then $a\in \Upl(K)$; and $\Upl(K)$ has an upper bound in $K$ but no least upper bound; these properties also hold for $\Upo(K)$ instead of $\Upl(K)$.
From Chapter~16 of~\cite{ADH} we need that
$T^{\text{nl}}$ has a certain extension by definitions
$T^{\text{nl}}_{\Upl\Upo}$ that has QE:
the language of $T^{\text{nl}}_{\Upl\Upo}$ is
$\mathcal{L}$ augmented by two extra binary relation symbols
$R_{\Upl}$ and $R_{\Upo}$, to be interpreted in~$K$ according to
$$ aR_{\Upl} b\Longleftrightarrow a\in \Upl(K)b, \qquad aR_{\Upo} b\Longleftrightarrow a\in \Upo(K)b.$$
(The language of $T^{\text{nl}}_{\Upl\Upo}$ in \cite[Chapter~16]{ADH} is slightly different, but yields the same notion of what is quantifier-free definable. The version here is more convenient for our purpose.)
Using that 
$\Upl(K)$ and $\Upo(K)$ are open-and-closed in $K$, it is routine 
(but tedious) to check that $K$ satisfies the differential analogue of \cite[2.15]{vdD89} that is discussed on p.~203 of that paper in a general setting. Thus:

\begin{cor} $K$ is $\d$-bounded; in particular, $\T$ is $\d$-bounded.
\end{cor}

\noindent
Moreover, \cite[p.~203]{vdD89} points out the following consequence (extending Corollary~\ref{defdimn}):

\begin{cor}\label{intko} Every nonempty definable set $S\subseteq K^n$ has nonempty interior in the Kolchin closure $S^{\Ko}$ of $S$ in $K^n$.
\end{cor} 

\noindent
(By our earlier convention, the {\em interior\/} here refers to the topology on $S^{\Ko}$ induced by the product topology on $K^n$
that comes from the order topology on $K$.) For nonempty
definable  $S\subseteq K^n$ with closure $\operatorname{cl}(S)$ in $K^n$ we have
$$\dim\!\big(\!\operatorname{cl}(S)\setminus S\big)\  <\  \dim S.$$
This is analogous to  \cite[2.23]{vdD89}, but the proof there doesn't go
through. We intend to show this dimension decrease 
in a follow-up paper.



\section{Dimension $0$ = Discrete}\label{d0d}

\noindent
Let $K$ be a Liouville closed $\upo$-free newtonian $H$-field, with the order topology on~$K$ and the corresponding product topology on each $K^n$. Corollary~16.6.11 in \cite{ADH} and its
proof yields the following equivalences for definable $S\subseteq K$:

\medskip\noindent
$$\dim S =0\ \Longleftrightarrow\ 
 \text{$S$ has empty interior}\ \Longleftrightarrow\  \text{$S$ is discrete.}$$

\medskip\noindent
We now extend part of this to definable subsets of $K^n$. The proof
of one of the directions is rather curious and makes full use of
the resources of \cite{ADH}.

\begin{prop} For definable nonempty $S\subseteq K^n$:
$$ \dim S =0\ \Longleftrightarrow\  \text{$S$ is discrete.}$$
\end{prop}
\begin{proof} For $i=1,\dots, n$ we let
$\pi_i\colon K^n\to K$ be given by $\pi_i(a_1,\dots, a_n)=a_i$. 
If $\dim S =0$, then $\dim \pi_i(S)=0$ for all $i$, so $\pi_i(S)$
is discrete for all $i$, hence the cartesian product $\pi_1(S)\times \cdots \times \pi_n(S)\subseteq K^n$
is discrete, and so is its subset $S$.

Now for the converse. Assume $S\subseteq K^n$ is discrete. We first replace $K$ by a suitable countable elementary substructure over which $S$ is defined and $S$ by its corresponding trace.
Now that $K$ is countable we next pass to its completion $K^{\operatorname{c}}$ as defined in~\cite[Section~4.4]{ADH}, which by~\cite[14.1.6]{ADH} is
an elementary extension of $K$. Replacing $K$ by $K^{\operatorname{c}}$
and $S$ by the corresponding
extension, the overall effect is that we have arranged
$K$ to be {\em uncountable,}\/ but with a {\em countable}\/ base for its topology. 
Then the discrete set $S$ is countable, so $\pi_i(S)\subseteq K$
is countable for each $i$, hence with empty interior, so
$\dim \pi_i(S)=0$ for all $i$, and thus
$\dim S =0$.  
\end{proof}

\begin{cor} If $S\subseteq K^n$ is definable and discrete, then
there is a neighborhood~$U$ of $0\in K^n$ such that $(s_1+U)\cap (s_2+U)=\emptyset$ for all distinct $s_1, s_2\in S$.
\end{cor} 

\begin{proof} Let $S\subseteq K^n$ ($n\ge 1$) be nonempty, definable, and discrete. For $y\in K^n$ we set
$|y|:= \max_i |y_i|$. The set $D:=\big\{|a-b|:\ a,b\in S\big\}$
is the image of a definable map $S^2 \to K$, so
$D$ is definable with $\dim D = 0$ and $0\in D$. Thus~$D$ is 
discrete, so $(-\epsilon, \epsilon)\cap D= \{0\}$ for some $\epsilon\in K^{>}$,
which gives the desired conclusion.
\end{proof}

\noindent
In particular, any definable discrete subset of $K^n$ is closed in $K^n$.

\section{Parametrizability by Constants}\label{sec:pc}

\noindent
Let $K$ be a Liouville closed $\upo$-free newtonian $H$-field. 
Then $K$ induces on its constant field $C$ just $C$'s structure as 
a real closed field, by \cite[16.0.2(ii)]{ADH}, that is, a set $X\subseteq C^m$
is definable in $K$ iff $X$ is semialgebraic in the sense of $C$. 

Let $S\subseteq K^n$ be definable. We say that 
$S$ is {\bf parametrizable by constants\/} if ${S\subseteq f(C^m)}$ 
for some $m$ and some definable map $f\colon C^m\to K^n$; equivalently,
$S=f(X)$ for some injective definable map $f\colon X\to K^n$
with semialgebraic $X\subseteq C^m$ for
some $m$. (The reduction to injective $f$
uses the fact mentioned above about the induced structure on $C$.) For example, if $P\in K\{Y\}$ is a differential polynomial of degree $1$ in a single indeterminate $Y$, then
the set $\big\{y\in K:\ P(y)=0\big\}$ is either empty or a translate of
a finite-dimensional $C$-linear subspace of $K$, and so this set
is parametrizable by constants. The definable
sets in $K^n$ for $n=0,1,2,\dots$ that are parametrizable by constants make up a very robust class: it is closed under taking definable
subsets, and under some basic logical operations: taking finite unions (in the same $K^n$), cartesian products, and images under definable maps. Moreover:

\begin{lemma}\label{pc} Let $S\subseteq K^n$ and  $f\colon S \to C^m$ be definable,
and let $e\in \N$ be such that $|f^{-1}(c)|\le e$ for all 
$c\in C^m$. Then $S$ is parametrizable by constants.
\end{lemma}
\begin{proof} By partitioning $S$ appropriately we reduce to
the case that for all $c\in f(S)$ we have $|f^{-1}(c)|= e$.
Using the lexicographic ordering on $K^n$ this yields definable
injective   $g_1,\dots, g_e\colon f(S)\to K^n$ such that
$f^{-1}(c)=\big\{g_1(c),\dots, g_e(c)\big\}$ for all $c\in f(S)$. 
Thus $S=g_1\big(f(S)\big)\cup \dots \cup g_e\big(f(S)\big)$ is parametrizable by constants.  
\end{proof}

\noindent
Suppose $S\subseteq K^n$ be definable.
Note that if $S$ is parametrizable by constants, then $\dim S \le 0$. The question arises if the converse holds: does it follow from $\dim S = 0$ that $S$ is parametrizable by constants?
We show that the answer is negative for $K=\T$ and the set
$$\big\{y\in \T:\ yy''=(y')^2\big\}\ =\ \{a\ex^{bx}:\ a,b\in \R\}.$$
This set has dimension $0$ and we claim that it is not parametrizable by constants. (The map $(a,b)\mapsto a\ex^{bx}: \R^2\to \T$ would be a parametrization of 
this set by constants if $\exp$ were definable in $\T$; we return to this issue
at the end of this section.) 
To justify this claim we appeal to a special case of results from \cite{ADH1}:

\medskip\noindent
{\em For any finite set $A\subseteq \T$ there exists an automorphism of the differential field $\T$ over $A$ that is not the identity on $\{\ex^{bx}:\ b\in \R\}$}.

\medskip
\noindent
The claimed nonparametrizability by constants follows when we combine this fact with the observation
that if $f\colon \R^m \to \T$ is definable in $\T$, say over the finite set~${A\subseteq \T}$, then any automorphism
of the differential field $\T$ over $A$ fixes each real number, and so it fixes each value of the function $f$.

\medskip
\noindent
Below $Y$ is a single indeterminate, and for $P\in K\{Y\}$ we let $$\Zero(P)\ :=\ \big\{y\in K:\ P(y)=0\big\}.$$
Thus $\Zero\!\big(YY''-(Y')^2\big)=\{a\ex^{bx}: a,b\in \R\}$ for $K=\T$ and $YY''-(Y')^2$ has order~$2$. What about the parametrizability
of $\Zero(P)$ for $P$ of order $1$? 
In the next two lemmas we consider the special case 
$P(Y)=F(Y)Y'-G(Y)$ where  $F,G\in C[Y]^{\neq}$ have no common factor of positive degree.

\begin{lemma}\label{pc1}
If $\frac{F}{G} = c\frac{\partial R}{\partial Y}/R$ for some $c\in C^\times$, $R\in C(Y)^\times$, or $\frac{F}{G} = \frac{\partial R}{\partial Y}$ for some $R\in C(Y)^\times$,
then $\Zero(P)$ is parametrizable by constants.
\end{lemma}
\begin{proof}
Suppose $\frac{F}{G} = c\frac{\partial R}{\partial Y}/R$ where  $c\in C^\times$, $R\in C(Y)^\times$.
Since $K$ is Liouville closed we can take~$b\in K^\times$ with $b^\dagger=1/c$.
Set $S:=\big\{y\in\Zero(P): G(y)\neq 0,\ R(y)\neq 0,\infty\big\}$.
Then for $y\in S$
we have $$0\ =\ G(y)\left(\textstyle\frac{F(y)}{G(y)}\,y'-1\right)\ =\
G(y) \left(c\big(\textstyle\frac{\partial R}{\partial Y}/R\big)(y)\,y'-1\right)\ =\ G(y)\big(cR(y)^\dagger-1\big)$$
and so $R(y)\in C^\times b$. It is clear that we can take $e\in \N$ such that the definable map $f\colon S\to C$ given by $f(y):=R(y)/b$
for $y\in S$ satisfies $|f^{-1}(c)|\leq e$ for all $c\in C$.
Hence $S$, and thus $\Zero(P)$, is parametrizable by constants by Lemma~\ref{pc}.
Next, suppose that $\frac{F}{G} = \frac{\partial R}{\partial Y}$ where $R\in C(Y)$. Take $x\in K$ with $x'=1$ and
set $S:=\big\{y\in\Zero(P): G(y)\neq 0,\ R(y)\neq \infty\big\}$. 
As before we obtain for $y\in S$ that $R(y)\in x+C$, and so~$\Zero(P)$ is parametrizable by constants.
\end{proof}

\noindent
Let $Q\in K\{Y\}$ be irreducible and let $a$ be an element of a differential field extension of $K$ with minimal annihilator~$Q$ over $K$. We say that $Q$  {\bf creates a constant} if
$C_{K\<a\>}\neq C$.
(This is related to the concept of ``nonorthogonality to the constants'' in the model theory of differential fields; see
\cite[Proposition~2.6]{McGrail}.) Note that our 
 $P=F(Y)Y'-G(Y)$ is irreducible in $K\{Y\}$.
 
\begin{lemma}\label{lem:Rosenlicht}
$P$ creates a constant iff 
$\frac{F}{G} = c\frac{\partial R}{\partial Y}/R$ for some $c\in C^\times$, $R\in C(Y)^\times$, or $\frac{F}{G} = \frac{\partial R}{\partial Y}$ for some $R\in C(Y)^\times$.
\end{lemma}
\begin{proof}
The forward direction holds by Rosenlicht~\cite[Proposition~2]{Rosenlicht74}. For the backward direction, take an element $a$  of a differential field extension of $K$ with minimal annihilator~$P$ over $K$. Consider first the case 
$\frac{F}{G} = c\frac{\partial R}{\partial Y}/R$ where $c\in C^\times$ and $R\in C(Y)^\times$. Take $b\in K^\times$ with $b^\dagger=1/c$.
As in the proof of Lemma~\ref{pc1} we obtain 
$0 = P(a) = G(a)\big(cR(a)^\dagger-1\big)$
with $G(a)\ne 0$, and thus $R(a)/b\in C_{K\<a\>}$ and
$R(a)/b\notin K$. The case 
$\frac{F}{G}=\frac{\partial R}{\partial Y}$
with $R\in C(Y)^\times$ is handled likewise. 
\end{proof}

\noindent
The following proposition therefore generalizes Lemma~\ref{pc1}:

\begin{prop}\label{pr:par1} If $P\in K\{Y\}$ is irreducible of order $1$ and creates a constant, then 
$\Zero(P)$ is parametrizable by constants.
\end{prop}

\noindent
Before we give the proof of this proposition, we prove two lemmas, in both of which we let
$P\in K\{Y\}$ be irreducible of order~$1$ such that $\Zero(P)$ is infinite.

\begin{lemma}
Let $Q\in K[Y,Y']\subseteq K\{Y\}$. Then $\Zero(P)\subseteq \Zero(Q)$ iff $Q\in PK[Y,Y']$.
\end{lemma}
\begin{proof}
Suppose $\Zero(P)\subseteq \Zero(Q)$ but $Q\notin PK[Y,Y']$. Put $F:=K(Y)$. By Gauss' Lemma, 
$P$ viewed as element of $F[Y']$ is irreducible and $Q\notin PF[Y']$. Thus there are   $A,B\in K[Y,Y']$, $D\in K[Y]^{\neq}$ with $D=AP+BQ$. Then $\Zero(P)\subseteq \Zero(D)$ is finite, a contradiction.
\end{proof}

\begin{lemma}
There is an element $a$ in an elementary extension of $K$ with minimal annihilator $P$ over $K$. 
\end{lemma}
\begin{proof}
Given $Q_1,\dots,Q_n\in K[Y,Y']^{\ne}$ with $\deg_{Y'}Q_i<\deg_{Y'}P$ for $i=1,\dots,n$, the previous lemma applied to $Q:=Q_1\cdots Q_n$ yields some $y\in K$ with $P(y)=0$ and $Q_i(y)\neq 0$ for all $i=1,\dots,n$.
 Now use compactness. 
\end{proof}

\begin{proof}[Proof of Proposition~\ref{pr:par1}] We can assume that $S:=\Zero(P)$ 
is infinite. The preceding lemma yields  an element $a$ in an elementary extension of $K$ with~${P(a)=0}$ and
$Q(a)\ne 0$ for all $Q\in K[Y, Y']^{\ne}$ with 
$\deg_{Y'}Q< d:=\deg_{Y'}P$. In particular,~$a$ is 
transcendental over $K$. Since $P$ creates a constant, $K\<a\>=K(a,a')$ has a constant~$c\notin C$. We have $c=A(a)/B(a)$ with 
$A\in K[Y, Y']$, $\deg_{Y'} A < d$,
$B\in K[Y]^{\ne}$. From $c'=0$ we get $A'(a)B(a)-A(a)B'(a)=0$, so
$$A'(Y)B(Y)-A(Y)B'(Y)\ =\ D(Y)P(Y)\ \text{ in }K\{Y\}\ 
\text{ with }\ D\in K[Y].$$
Hence for $y\in S$ with $B(y)\ne 0$ we have $\big(A(y)/B(y)\big)'=0$, 
that is, $A(y)/B(y)\in C$.  
Thus for $S_B:=\big\{y\in S:\ B(y)\ne 0\big\}$ we have a definable map
$$f\colon S_B \to C, \qquad f(y)\ :=\ A(y)/B(y).$$ 
Since $c$ is transcendental over $K$, $a$ is algebraic over $K(c)$, say 
$$F_0(c)a^e+ F_1(c)a^{e-1} + \cdots + F_e(c)\ =\ 0,$$
where $F_0,F_1,\dots, F_e\in K[Z]$ have no common divisor of 
positive degree in $K[Z]$. Let $G:=\partial P/\partial Y'$ be 
the separant of $P$. Then $G(a)\ne 0$, 
$K\big[a, a', 1/B(a), 1/G(a)\big]$ is a differential subring of $K(a,a')$,
and every $y\in S_B$ with $G(y)\ne 0$ yields a differential ring
morphism $$\phi_y\ :\  K\big[a, a', 1/B(a), 1/G(a)\big] \to K$$ that is 
the identity on $K$
with $\phi_y(a)=y$; see the subsection on minimal annihilators
in \cite[Section~4.1]{ADH}. Moreover, 
$c=A(a)/B(a)\in K\big[a, a', 1/B(a), 1/G(a)\big]$, and so for 
$y\in S_B$ with $G(y)\ne 0$ we have $\phi_y(c)=A(y)/B(y)=f(y)$, so
$$F_0\big(f(y)\big)y^e + F_1\big(f(y)\big)y^{e-1}+ \cdots + F_e\big(f(y)\big)\ =\ 0.$$ 
Set $S_{B,G}:=\big\{y\in S_B:\ G(y)\ne 0\big\}$. Then
$S\setminus S_{B,G}$ is finite, and the above shows that
for all $z\in f(S_{B,G})$ we have $|f^{-1}(z)\cap S_{B,G}|\le e$.
Now use Lemma~\ref{pc}. 
\end{proof}

\noindent
Freitag~\cite{Freitag} proves a generalization of Lemma~\ref{lem:Rosenlicht}.
Nishioka (\cite{Nishioka}, see also \cite[p.~90]{Matsuda}) gives sufficient conditions on irreducible differential polynomials of order~$1$ to create a constant,
involving the concept of ``having no movable singularities''; this can be used to give
further examples of $P\in K\{Y\}$ of order~$1$ whose zero set is parametrizable by constants.
But we do not know whether~$\Zero(P)$ is parametrizable by constants for every $P\in K\{Y\}$ of order $1$.

\subsection*{Open problems} The definable set 
$$\big\{y\in \T:\ yy''=(y')^2\big\}\ =\ \{a\ex^{bx}:\ a,b\in \R\}\subseteq \T^2$$ is the image of the map $(a,b) \mapsto a\ex^{bx}\colon\ \R^2\to \T^2$,
and so by the above negative result this map is not definable in the differential field
$\T$. But it is definable in the {\em exponential\/} differential field~$(\T, \exp)$, where exponentiation on $\T$ is taken
as an extra primitive. This raises the question whether parametrizability by constants holds in an extended sense where
the parametrizing maps are allowed to be definable in $(\T, \exp)$. More precisely, {\em if $S\subseteq \T^n$ is definable in $\T$ with
$\dim S = 0$, does there always exist an $m$ and a map $f\colon \R^m \to \T^n$, definable in $(\T, \exp)$, with
$S\subseteq f(\R^n)$}\/? (It is enough to have this for $n=1$ and $S=\big\{y\in \T:\ P(y)=0\big\}$, $P\in\T\{Y\}^{\ne}$.) 

\medskip\noindent
This is of course related to the issue whether the results in
\cite[Chapter~16]{ADH} about~$\T$ generalize to its expansion
$(\T, \exp)$. In particular, 
{\em is the structure induced on $\R$ by~$(\T, \exp)$ 
just the exponential field structure of $\R$}?



It would be good to know more about the order types of discrete definable subsets of 
Liouville closed $\upo$-free newtonian $H$-fields $K$. For example, 
can any such set have order type $\omega$, or more generally,
 have an initial segment of order type $\omega$?


 
\section{Dimension $0$~=~Co-Analyzable Relative to the Constant Field}\label{sec:ca}

\noindent
Parametrizability by constants was our first guess
of the model-theoretic significance of \cite[Theorem 16.0.3]{ADH} which says that a Liouville closed $\upo$-free newtonian $H$-field has no proper differentially-algebraic $H$-field extension
with the same constants. 
 As we saw, this guess failed on the set of zeros of
$YY''-(Y')^2$. We subsequently realized that
the notion of {\em co-analyzability\/} from~\cite{HHM} fits exactly our situation. 
Below we expose what we need from that paper, and next we apply it to $\T$.

\subsection*{Co-analyzability}
We adopt here the model-theory notations of~\cite[Appendix~B]{ADH}. Let $\mathcal L$ be a first-order language with a distinguished unary relation symbol~$C$. For convenience we assume $\mathcal L$ is {\em one-sorted}. Let $\mathbf M=(M;\dots)$ be an $\mathcal L$-structure and let $C^{\mathbf M}\subseteq M$ (or just $C$ if
$\mathbf M$ is clear from the context) be the interpretation of the symbol $C$ in $\mathbf M$; we assume $C\ne \emptyset$.  

\medskip\noindent 
Assume $\mathbf M$ is $\omega$-saturated.
Let $S\subseteq M^n$ be definable. 
By recursion on $r\in \N$ we 
define what makes $S$ {\bf co-analyzable in
$r$ steps} (tacitly: relative to $\mathbf M$ and $C$):
\begin{enumerate}
\item[(C$_{0}$)] $S$ is co-analyzable in $0$ steps iff $S$ is finite;
\item[(C$_{r+1}$)] $S$ is co-analyzable in $r+1$ steps iff for some definable set 
$R\subseteq C\times M^n$,
\begin{enumerate}
\item the natural projection $C\times M^n \to M^n$ maps $R$ onto
$S$;
\item for each $c\in C$, the section $R(c):=\big\{s\in M^n:(c,s)\in R\big\}$ above~$c$ is
co-analyzable  in $r$ steps. \end{enumerate}
\end{enumerate}
We call $S$ {\bf co-analyzable} if $S$ is co-analyzable  in $r$ steps for some~$r$.

\medskip\noindent
Thus in (C$_{r+1}$) the set $R$ gives rise to a covering $S=\bigcup_{c\in C} R(c)$
of $S$ by definable sets $R(c)$ that are co-analyzable in $r$ steps. Of course, the definable set $C^r\subseteq M^r$ is the archetype of a definable set that is co-analyzable in $r$ steps. 
Note that if $S$ is co-analyzable in $1$ step, then
the $\omega$-saturation of $\mathbf M$ yields
for $R$ as in (C$_1$) a uniform bound $e\in \N$ such that $|R(c)|\le e$ for all $c\in C$. This $\omega$-saturation gives likewise an automatic uniformity in (C$_{r+1}$) that enables us to extend
the notion of co-analyzability appropriately to arbitrary $\mathbf M$ (not necessarily
$\omega$-saturated). Before doing this, we mention some easy consequences of the definition above where
we do assume $\mathbf M$ is $\omega$-saturated. First, if the definable
set $S\subseteq M^n$ is co-analyzable in $r$ steps, then $S$ is co-analyzable in $r+1$ steps: use induction on $r$.
Second, if the definable
set $S\subseteq M^n$ is co-analyzable in $r$ steps, then so is any definable subset of $S$, and the image $f(S)$ under any definable
map $f\colon S\to M^m$. Third, if the definable sets
$S_1, S_2\subseteq M^n$ are co-analyzable in $r_1$ and $r_2$ steps, respectively, then $S_1\cup S_2$ is
co-analyzable in $\max(r_1,r_2)$ steps. Finally, if
the definable sets 
$S_1\subseteq M^{n_1}$ and $S_2\subseteq M^{n_2}$ are co-analyzable in $r_1$ steps and $r_2$ steps, respectively,
then $S_1\times S_2\subseteq M^{n_1+n_2}$ is co-analyzable
in $r_1+r_2$ steps. In any case, the class of co-analyzable definable sets
is clearly very robust.

\medskip\noindent
Next we extend the notion above to arbitrary $\mathbf M$, not necessarily $\omega$-saturated. Let $S\subseteq M^n$ be definable. Define an {\bf $r$-step co-analysis of $S$\/} by recursion
on $r\in \N$ as follows: for $r=0$ it is an $e\in \N$ with $|S|\le e$.
For $r=1$ it is a tuple $(e,R)$ with~$e\in \N$ and definable
$R\subseteq C\times M^n$ such that the natural projection
$C\times M^n\to M^n$ maps~$R$ onto $S$, and $|R(c)|\le e$ for
all $c\in C$. Given $r\ge 1$, an $(r+1)$-step co-analysis of
$S$ is a tuple $(e,R_1,\dots,R_{r+1})$ with $e\in\N$ and definable 
sets $$R_i\subseteq  C\times M^n \times M^{d_{i}} \times \cdots\times M^{d_{r}} \quad (i=1,\dots,r+1,\ d_1,\dots,d_{r}\in\N),$$
(so $R_{r+1}\subseteq C\times M^n$), 
such that the natural projection $C\times M^n  \to M^n$ maps $R_{r+1}$ onto $S$,
and for each $c\in C$ there exists $b\in M^{d_{r}}$ for which
the tuple $\big(e,R_1^b,\dots,R_{r}^b\big)$ is an $r$-step co-analysis of $R_{r+1}(c)\subseteq S$. (Here we use the following notation for a relation $R\subseteq P\times Q$: for $q\in Q$ we
set $R^q:=\{p\in P:\ (p,q)\in R\}$.)

For model-theoretic use the reader should note the following 
uniformity with respect to parameters from $M^m$: let $e, R_1,\dots, R_{r+1},S$ be given with
$e\in \N$, $0$-definable $R_i\subseteq M^m\times C \times M^{d_i}\times \cdots \times M^{d_r}$ for $i=1,\dots,r+1$, and $0$-definable $S\subseteq M^m\times M^n$. Then the set of $a\in M^m$ such that $\big(e, R_1(a),\dots, R_{r+1}(a)\big)$ is an $(r+1)$-step co-analysis of $S(a)$ is $0$-definable. 
Moreover, one can take a defining $\mathcal{L}$-formula for this subset of
$M^m$ that depends only on $e$ and given defining $\mathcal{L}$-formulas for 
$R_1,\dots, R_{r+1},S$, not on $\mathbf M$.  

\medskip\noindent
If $\mathbf M$ is $\omega$-saturated, then a definable set $S\subseteq M^n$ 
can be shown to be co-analyzable
in $r$ steps iff there exists an $r$-step co-analysis of $S$.
(To go from co-analyzable in~$r$ steps to an $r$-step co-analysis requires the
uniformity noted above.) Thus for arbitrary $\mathbf M$ and definable
$S\subseteq M^n$ we can define without ambiguity $S$ to be co-analyzable in 
$r$ steps
if there exists an $r$-step co-analysis of $S$; likewise, $S$ is defined to 
be co-analyzable if $S$ is co-analyzable in $r$ steps for some $r$. 
After the proof of Lemma~\ref{ca=fib} we give an example of a definable 
$S\subseteq \T$ that is co-analyzable in $2$ steps but not in $1$ step 
(relative to $\T$ and $\R$).

\medskip
\noindent
Let $S\subseteq M^n$ be definable and $\mathbf M^*$
an elementary extension of $\mathbf M$. We denote by 
$S^*\subseteq (M^*)^n$ the extension of $S$ to $\mathbf M^*$: choose 
 an $\mathcal L_M$-formula $\varphi(x)$, where $x=(x_1,\dots,x_n)$, with $S=\varphi^{\mathbf M}$, and set
 $S^*:=\varphi^{\mathbf M^*}$. Then for a tuple $(e,R_1,\dots, R_{r+1})$ with $e,r\in \N$ and definable $R_i\subseteq  C\times M^n \times M^{d_{i}} \times \cdots\times M^{d_{r}}$ for 
 $i=1,\dots,r+1$ we have:
 $(e, R_1,\dots, R_{r+1})$ is an $(r+1)$-step co-analysis of $S$ iff
$(e, R_1^*,\dots, R_{r+1}^*)$ is an $(r+1)$-step co-analysis of $S^*$. 
Here is \cite[Proposition 2.4]{HHM}:  

\begin{prop}\label{hmm} 
Let the language $\mathcal L$ be countable and let $T$ be a complete $\mathcal L$-theory such that
$T\vdash \exists x C(x)$.  
Then the following conditions on an $\mathcal L$-formula $\varphi(x)$ with $x=(x_1,\dots,x_n)$ are equivalent:
\begin{enumerate}
\item[(i)] for some  model $\mathbf M$ of $T$, $\varphi^{\mathbf M}$ is co-analyzable;
\item[(ii)] for every model $\mathbf M$ of $T$, $\varphi^{\mathbf M}$ is co-analyzable;
\item[(iii)] for every model $\mathbf M$ of $T$, if $C^{\mathbf M}$ is countable, then so is $\varphi^{\mathbf M}$;
\item[(iv)] for all models $\mathbf M\preceq\mathbf M^*$ of $T$, if
$C^{\mathbf M}=C^{\mathbf M^*}$, then $\varphi^{\mathbf M}=\varphi^{\mathbf M^*}$.
\end{enumerate}
\end{prop}

\noindent
The equivalence (i)~$\Leftrightarrow$~(ii) and the implication
(ii)~$\Rightarrow$~(iii) are clear from the above, and (iii)~$\Rightarrow$~(iv)
holds by Vaught's two-cardinal theorem~\cite[Theorem 12.1.1]{Hodges}. The contrapositive of 
(iv)~$\Rightarrow$~(i) is obtained in \cite{HHM} by an omitting types argument. 

\subsection*{Application to $\T$} Let $\mathcal{L}$ be the language of ordered valued differential fields from Section~\ref{T}, except that we consider it as having in addition a distinguished unary relation symbol $C$; an $H$-field is construed
as an $\mathcal{L}$-structure as before, 
with~$C$ in addition interpreted as its constant field.

Let $K$ be a Liouville closed $\upo$-free newtonian $H$-field
and $P\in K\{Y\}^{\ne}$. If $K\preceq K^*$ and
$K$ and $K^*$ have the same constants, then $P$ has the same zeros in $K$ and $K^*$, by \cite[Theorem 16.0.3]{ADH}. Thus the zero
set $\Zero(P)\subseteq K$ is co-analyzable by Proposition~\ref{hmm}
applied to the $\mathcal{L}_A$-theory $T:=\text{Th}(K_A)$ where $A$ is the 
finite set of nonzero coefficients of $P$. In fact:

\begin{prop} Let $S\subseteq K^n$ be definable, $S\ne \emptyset$. Then
$$S \text{ is co-analyzable }\ \Longleftrightarrow\ 
\dim S\ =\ 0.$$
\end{prop}
\begin{proof} Suppose $\dim S=0$. Then for $i=1,\dots, n$
and the $i$th coordinate projection $\pi_i\colon K^n \to K$
we have $\dim \pi_i(S)=0$, and thus $\pi_i(S)\subseteq \Zero(P_i)$ with 
$P_i\in K\{Y\}^{\ne}$. Since each $\Zero(P_i)$ is co-analyzable
and 
$S\subseteq \Zero(P_1)\times\cdots \times \Zero(P_n)$, we  conclude that~$S$ is co-analyzable. Conversely, assume that $S$ is co-analyzable, say in~$r$ steps. To get $\dim S = 0$ we can arrange that $K$ is
$\omega$-saturated. Using $\dim C=0$ and induction on~$r$
it follows easily from the behavior of dimension in definable families (Theorem~\ref{ddim}) that $\dim S=0$.
\end{proof} 

\noindent
Let $\dim_C S$ be the least $r\in \N$ such that
$S$ is co-analyzable in $r$ steps, for nonempty definable $S\subseteq K^n$
with $\dim S=0$ (and $\dim_C \emptyset := -\infty$). 
It is easy to show 
that $\dim_C S$ coincides with the usual semialgebraic dimension of $S$ 
(with respect to the real closed field $C$) when $S\subseteq C^n$ is
semialgebraic. In general, $\dim_C S$ behaves
much like a dimension function, and it would be good to confirm this
by showing for example that for definable $S_i\subseteq K^{n_i}$ with $\dim S_i=0$ for $i=1,2$ we have
$$\dim_C S_1\times S_2\ =\ \dim_C S_1 + \dim_C S_2.$$
(We do know that the quantity on the left is at most that on the right.)
Another question is whether $\dim_C \Zero(P)\le \text{order}(P)$ for $P\in K\{Y\}^{\ne}$.

\medskip\noindent
Towards the uniform finiteness property mentioned at the end of the introduction, we introduce a condition that is equivalent to
co-analyzability. 

Let $K$ be $\omega$-saturated and $S\subseteq K^n$ be definable.
By recursion on $r\in \N$ we define what makes $S$ {\bf fiberable
by $C$ in $r$ steps}: for $r=0$ it means that $S$ is finite; $S$ is fiberable by $C$ in $(r+1)$ steps iff there is a definable map $f\colon S \to C$ such that~$f^{-1}(c)$ is fiberable by $C$ in $r$ steps for every $c\in C$.

\begin{lemma}\label{ca=fib} $S$ is co-analyzable in $r$ steps iff $S$ is 
fiberable by $C$ in $r$ steps.
\end{lemma} 
\begin{proof} By induction on $r$. The case $r=0$ is trivial. Assume $S$ is co-analyzable in $(r+1)$ steps, so we have a definable $R\subseteq C\times K^n$ that is mapped onto $S$ under the natural projection $C\times K^n\to K^n$ and such that
$R(c)$ is co-analyzable in $r$ steps for all $r$.
For $s\in S$ the definable nonempty set $R^s\subseteq C$ is a finite union of intervals and points, and so we can pick a point
$f(s)\in R^s$ such that the resulting function $f\colon S \to C$
is definable. Then $f^{-1}(c)\subseteq R(c)$ is co-analyzable
in $r$ steps for all $c\in C$, and so fiberable by $C$ in
$r$ steps by the inductive assumption. Thus $f$ witnesses
that $S$ is fiberable by $C$ in $(r+1)$ steps.
The other direction is clear. 
\end{proof} 

\noindent
As an example, consider $S=\Zero\!\big(YY''-(Y')^2\big)$. Then we have a definable 
(surjective) function $f\colon S \to C$ given by $f(y)=y^\dagger$ for nonzero 
$y\in S$, and $f(0)=0$. For  $c\in C^\times$ we take any $y\in S$ with $f(y)=c$, and then $f^{-1}(c)=C^\times y$; also $f^{-1}(0)=C$.
Thus~$f$ witnesses that $S$ is fiberable by $C$ in two steps. Moreover,
$S$ is not fiberable by~$C$ in one step: if it were, Lemma~\ref{pc} would 
make $S$ parametrizable by constants, which we know is not the case.

An advantage of fiberability by $C$ over co-analyzability is that
for $f\colon S \to C$ and $R\subseteq C\times S$ witnessing these notions 
the fibers $f^{-1}(c)$ in
$S=\bigcup_c f^{-1}(c)$ are pairwise disjoint, which is not necessarily the case for the sections $R(c)$ in $S=\bigcup_c R(c)$. Below we use the equivalence
$$S \text{ is finite} \Longleftrightarrow\ f(S) \text{ is finite and every fiber $f^{-1}(c)$ is finite}.$$ 
to obtain the uniform finiteness property mentioned at the
end of the introduction. We state this property here again in a slightly 
different form, 
with $K$ any Liouville closed $\upo$-free newtonian $H$-field:

\begin{prop} Let $D\subseteq K^m$ and $S\subseteq D\times K^n$ be definable. Then there exists an
$e\in \N$ such that $|S(a)|\le e$ whenever $a\in D$ and $S(a)$ is finite.
\end{prop}  
\begin{proof} We first consider the special case that $n=1$ and 
$S(a)\subseteq C$ for all~${a\in D}$. By \cite[16.0.2(ii)]{ADH} a subset of $C$ is definable in $K$ iff it is semialgebraic in the sense of $C$.
Thus $S(a)$ is finite iff it doesn't contain
any interval $(b,c)$ in $C$ with~${b < c}$ in~$C$; the uniform bound
then follows by a routine 
compactness argument. Next we reduce the general case to this special case.

First, using Proposition~\ref{Kddim} we   shrink $D$ to arrange that
$\dim S(a)=0$ for all~${a\in D}$. Next, we arrange that $K$ is $\omega$-saturated,
so $S(a)$ is fiberable by~$C$ for every~${a\in D}$. Saturation allows us to reduce further to the case that for
a fixed~${r\in \N}$ every section $S(a)$ is fiberable by $C$ in $(r+1)$ steps.
We now proceed by induction on $r$. 
Model-theoretic compactness yields a definable function $f\colon S \to C$ such that
for every $a\in D$ the function $f_a\colon S(a) \to C$ given by $f_a(s)=f(a,s)$ 
witnesses that~$S(a)$ is fiberable by~$C$ in $(r+1)$ steps, that is,
$f_a^{-1}(c)$ is fiberable by~$C$ in $r$ steps for all $c\in C$.

Inductively we have $e\in \N$ such that
$|f_a^{-1}(c)|\le e$ whenever $a\in D, c\in C$, and~$f_a^{-1}(c)$ is finite.
The special case we did in the beginning of the proof gives~$d\in \N$
such that $|f_a(S(a))|\le d$ whenever $a\in D$ and $f_a(S(a))$ is finite.
 For $a\in D$ we have $S(a)=\bigcup_cf_a^{-1}(c)$, so if $S(a)$ is finite, then
$|S(a)|\le de$.  
\end{proof}

\noindent
To fully justify the use of saturation/model-theoretic compactness in the proof
above requires an explicit notion of ``$r$-step fibration by $C$'' (analogous to that of ``$r$-step co-analysis'') that makes sense 
for any $K$, not necessarily $\omega$-saturated. We leave this to the reader, and just note a nice
consequence: if $S\subseteq K^n$ is definable, infinite, and $\dim S =0$, 
then $S$ has the same cardinality as $C$. (This reduces to the fact that any infinite semialgebraic subset of $C$ has the same cardinality as $C$.)
In particular, there is no countably infinite definable set $S\subseteq \T$.
  
\medskip
\noindent
As an application of the material above we show
that the differential field $K$ does not eliminate 
imaginaries. More precisely: 

\begin{cor} No definable map $f\colon K^\times\to K^n$ is such that for all $a,b\in K^\times$,
$$a\asymp b\ \Longleftrightarrow\ f(a)=f(b).$$
\end{cor}
\begin{proof} By \cite[Lemmas~16.6.10,~14.5.10]{ADH} there exists an
elementary extension of $K$ with the same constant field $C$ as $K$ and whose
value group has greater cardinality than $C$. Suppose  $f\colon K^\times\to K^n$ is definable such that for all $a,b\in K^\times$ we have: 
$a\asymp b \Leftrightarrow f(a)=f(b)$. We can arrange that the value group of
$K$ has greater cardinality than $C$, and so $f(K^\times)\subseteq K^n$ has dimension $>0$. 
Every fiber $f^{-1}(p)$ with $p\in f(K^\times)$ is a nonempty open subset of
$K^\times$, so has dimension $1$, and thus $\dim K^\times >1$ by $\d$-boundedness
of $K$, a contradiction.
\end{proof}

\end{document}